\documentclass[11pt]{amsart}
\usepackage{latexsym,amscd,amsrefs,amssymb}
\usepackage{txfonts,mathptmx}
\usepackage{mathrsfs}

\numberwithin{equation}{section}

\newtheorem{thm}{Theorem}[section]
\newtheorem*{thm-nn}{Theorem}

\newtheorem{lem}[thm]{Lemma}
\newtheorem{pro}[thm]{Proposition}
\newtheorem*{pro-nn}{Proposition}
\newtheorem{cor}[thm]{Corollary}
\newtheorem*{cor-nn}{Corollary}
\newtheorem*{conj-nn}{Conjecture}

\theoremstyle{definition}

\theoremstyle{remark}

\newtheorem*{ack}{Acknowledgement}

\newcommand{\F}{\mathbb{F}}

\newcommand{\Q}{\mathbb{Q}}

\newcommand{\Z}{\mathbb{Z}}
\newcommand{\sph}{\mathbb{S}}

\newcommand{\C}{\mathbb{C}}

\newcommand{\rank}{{\operatorname{rank}}}

\newcommand{\signmod}{\F_p[-1]}

\begin{document}
\title[Vaninsing ranges for alternating subgroups]{
Vanishing ranges for the mod $p$ cohomology of alternating subgroups of
Coxeter groups}
\author[T. Akita]{
Toshiyuki Akita}
\author[Y. Liu]{Ye Liu}

\address{Department of Mathematics, Hokkaido University,
Sapporo, 060-0810 Japan}
\email{akita@math.sci.hokudai.ac.jp}
\email{liu@math.sci.hokudai.ac.jp}
\subjclass[2010]{Primary~20F55, 20J06; Secondary~55N91}

\maketitle

\begin{abstract}
We obtain vanishing ranges for the mod $p$ cohomology of alternating
subgroups of finite $p$-free Coxeter groups.
Here a Coxeter group $W$ is $p$-free if the order of the product $st$ is
prime to $p$ for every pair of Coxeter generators $s,t$ of $W$.
Our result generalizes those for alternating groups
formerly proved by Kleshchev-Nakano and Burichenko.
As a byproduct, we obtain vanishing ranges for
the twisted cohomology of finite $p$-free Coxeter groups
with coefficients in the sign representations.
In addition, a weak version of the main result is proved for
a certain class of infinite Coxeter groups.
\end{abstract}

\section{Introduction}
Let $(W,S)$ be a Coxeter system, the pair of a Coxeter group $W$ and the set $S$
of Coxeter generators of $W$.
The alternating subgroup $A_W$ of
$W$ is the kernel of the sign homomorphism
$W\rightarrow\{\pm 1\}$ defined by $s\mapsto -1$ $(s\in S)$.
The symmetric group $\Sigma_n$ on $n$ letters is a finite Coxeter group (of type $A_{n-1}$),
and its alternating subgroup is nothing but the alternating group $A_n$ on $n$ letters.
Cohomology of alternating groups has been studied by various authors.
Among others, 
Kleshchev-Nakano \cite{MR1875898}*{p.~354 Corollary} and
Burichenko \cite{MR2015284}*{Theorem 1.4}
independently obtained vanishing ranges for the mod $p$ cohomology
of alternating groups:
\begin{thm}[Kleshchev-Nakano, Burichenko]\label{thm-alt}
Let $p$ be an odd prime.
The mod $p$ cohomology $H^{k}(A_{n},\F_{p})$
vanishes for $0<k<p-2$.
\end{thm}
The primarly purpose of this paper is to generalize Theorem \ref{thm-alt}
to alternating subgroups of finite Coxeter groups, and thereby to give an
alternative proof of Theorem \ref{thm-alt}.
Our main result is the following:
\begin{thm}\label{main-thm}
Let $p$ be an odd prime,
$W$ a finite $p$-free Coxeter group, and $A_W$ the alternating subgroup
of $W$.
Then the mod $p$ cohomology $H^{k}(A_W,\F_{p})$
vanishes for $0<k< p-2$.
\end{thm}
Here a Coxeter group is $p$-free if the order of the product $st\in W$ is prime
to $p$ for every pair of Coxeter generators $s,t\in S$.
Since symmetric groups $\Sigma_n$ are $p$-free for $p\geq 5$ as Coxeter groups,
Theorem \ref{thm-alt} is a special case of our theorem.
Note that finiteness and $p$-freeness assumptions on $W$ are necessary, 
and vanishing ranges for $H^{k}(A_W,\F_{p})$ are best possible. 
See \S\ref{subsec:alt}, \S\ref{sec:twisted} and
\S\ref{sec:infinite} for precise.

The key ingredients for the proof of Theorem \ref{main-thm} are
\begin{enumerate}
\item the classification of finite Coxeter groups (see \S\ref{sec-casebycase}),
\item high connectivity of the Coxeter complex $X_W$ (Proposition \ref{pro:Coxeter-cpx},
\ref{serre-cont}),
\item high connectivity of the orbit space $X_W/A_W$ (Proposition \ref{orbit-space}),
\end{enumerate}
as well as some considerations of equivaraint cohomology in \S\ref{sec:equiv-coho}.
The proof is inspired by the arguments in
\cite{preprint-coxeter-p-local}, where the first author obtained vanishing ranges
for the $p$-local homology of $p$-free Coxeter groups.
As a byproduct of Theorem \ref{main-thm}, we will obtain
vanishing ranges for the twisted cohomology of finite $p$-free Coxeter groups
with coefficients in the sign representations over $\F_p$
(Theorem \ref{thm:signmod}).
As we remarked above, Theorem \ref{main-thm} no longer holds for
infinite Coxeter groups.
Instead, we will prove a weak version
of Theorem \ref{main-thm} for a certain class of infinite Coxeter groups
(Theorem \ref{thm:infinite}).

\section{Preliminaries}\label{sec-defn}
\subsection{Coxeter groups}\label{subsec-defn}
In this subsection, 
we recall definitions and relevant facts concerning of Coxeter groups.
References are \cites{book-building,bourbaki,humphreys}.
Let $S$ be a finite set.
A Coxeter matrix is a symmetric matrix $M=(m(s,t))_{s,t\in S}$
each of whose entries $m(s,t)$ is a positive integer or $\infty$ such that
\begin{enumerate}
\item $m(s,s)=1$ for all $s\in S$,
\item $2\leq m(s,t)=m(t,s)\leq\infty$ for all distinct $s,t\in S$.
\end{enumerate}
The \emph{Coxeter system} associated to $M$ is the pair $(W,S)$
where $W$ is the group generated by $s\in S$
and the fundamental relations $(st)^{m(s,t)}=1$ $(m(s,t)<\infty)$:
\[
W:=\langle s\in S\ |\ (st)^{m(s,t)}=1 (m(s,t)<\infty)\rangle.
\]
The group $W$ is called the \emph{Coxeter group} associated to $M$,
elements of $S$ are called \emph{Coxeter generators} of $W$,
and the cardinality of $S$ is called the \emph{rank} of $W$
and is denoted by $|S|$ or $\rank\,W$.

For a subset $T\subseteq S$,
the subgroup $W_T:=\langle T\rangle$ of $W$ generated
by elements $t\in T$ is called a (standard) \emph{parabolic subgroup}
(or a special subgroup in the literature).
In particular, $W_S=W$ and $W_\varnothing=\{1\}$.
It is known that $(W_T,T)$ is a Coxeter system
associated to the restriction of the Coxeter matrix to $T$.
Henceforth, we sometimes omit the reference to the Coxeter matrix $M$
and the set of Coxeter generators $S$ 
if there is no ambiguity.

Finally, given an odd prime number $p$,
a Coxeter group $W$ is called \emph{$p$-free} in \cite{preprint-coxeter-p-local}
if $m(s,t)$ is prime to $p$
for every pair of Coxeter generators $s,t\in S$. By convention, $\infty$ is prime to all prime numbers.
Since the order of the product $st$ $(s,t\in S)$ is precisely $m(s,t)$, the definition
agrees with the one given in the introduction.
For every finite irreducible Coxeter group $W$, 
the range of odd prime numbers $p$ such that $W$ is $p$-free can be found in Appendix
(see \S\ref{sec-casebycase} for the definition of irreducible Coxeter groups).

\subsection{Known results for cohomology of alternating subgroups}\label{subsec:alt}
As was defined in the introduction, the \emph{alternating subgroup} $A_W$ of a Coxeter
group of $W$ is the kernel of the sign homomorphism $W\rightarrow\{\pm 1\}$
which assigns $-1$ to $s\in S$.
Alternating subgroups are also called rotation(al) subgroups
in the literature.
To the best of our knowledge, 
not much is known about cohomology of alternating subgroups of Coxeter groups.
Finite presentations of alternating subgroups of arbitrary Coxeter groups were
given in Bourbaki \cite{bourbaki}*{Chapitre IV \S1, Exercise 9} as an
exercise.
The proof can be found in
Brenti-Reiner-Roichman \cite{MR2417024}*{Proposition 2.1.1}.
From finite presentations, one can compute the first integral homology
of alternating subgroups.
Moreover, Maxwell \cite{MR0486097} determined the Schur multiplier
$H^2(A_W,\C^\times)\cong H_2(A_W,\Z)$ for alternating subgroups $A_W$
of {\em finite} Coxeter groups $W$.

Before preceeding further,
we remark on the $p$-freeness assumption in Theorem  \ref{main-thm}.
Let
$
D_{2m}:=\langle s,t\mid s^2=t^2=(st)^m=1\rangle
$
be the dihedral group of order $2m$, which is a Coxeter group of rank $2$.
It is $p$-free if and only if $p$ does not divide $m$.
The alternating subgroup of $D_{2m}$ is the cyclic group of order $m$ generated
by $st$.
Now suppose that $p$ divides $m$. 
It is well-known that $H^*(\Z/m,\F_p)\cong \F_p[u]\otimes E(v)$,
where $\F_p[u]$ is the polynomial algebra 
generated by a two dimensional generator $u$
and $E (v)$ is the exterior algebra generated by a one dimensional
generator $v$.
In particular, $H^1(\Z/m,\F_p)\cong\F_p$,
which shows the necessity of $p$-freeness assumption
in Theorem \ref{main-thm}.

\section{The case $|S|\leq p-2$}\label{sec-casebycase}
In this section, we will prove Theorem \ref{main-thm} for $|S|\leq p-2$
by using the classification of finite Coxeter groups.
Recall that the Coxeter matrix $M=(m(s,t))_{s,t\in S}$ defining the
Coxeter system $(W,S)$ is represented by the 
\emph{Coxeter graph} $\Gamma$ whose
vertex set is $S$ and whose edges are the unordered pairs $\{s,t\}\subset S$
with $m(s,t)\geq 3$. The edges $\{s,t\}$ with $m(s,t)\geq 4$ are labeled
by those numbers.
For convenience, we write $W=W(\Gamma)$
and call it the Coxeter group of type $\Gamma$.
A Coxeter group $W(\Gamma)$ is called \emph{irreducible} if 
$\Gamma$ is connected,
otherwise called \emph{reducible}.
For a reducible Coxeter group $W(\Gamma)$, 
if $\Gamma$ consists of the connected components
$\Gamma_1,\Gamma_2,\dots,\Gamma_r$, then $W(\Gamma)$
is isomorphic to the internal direct product of parabolic subgroups
$W(\Gamma_i)$'s
each of which is irreducible:
\[
W(\Gamma)=W(\Gamma_1)\times W(\Gamma_2)\times\cdots\times W(\Gamma_r).
\]

The classification of Coxeter graphs for finite irreducible Coxeter groups is well-known.
They consist of
infinite families $A_n\, (n\geq 1)$, $B_n\, (n\geq 2)$, $D_n\, (n\geq 4)$,
$I_2(m)\, (m\geq 3)$, and exceptional graphs 
$H_3,H_4,F_4,E_6,E_7$ and $E_8$.
The subscript stands for the rank of the corresponding Coxeter group.
See Appendix for the orders of finite irreducible Coxeter groups.
Note that $W(A_n)$ is isomorphic to the symmetric group of $n+1$ letters,
and $W(I_2(m))$ is isomorphic to the dihedral group of order $2m$.

\begin{pro}\label{low-rank}
Let $p\geq 5$ be a prime and
$W$ a finite $p$-free Coxeter group with $|S|\leq p-2$.
Then $W$ has no $p$-torsion.
\end{pro}
\begin{proof}
If $W$ is a finite $p$-free Coxeter group, then $W$ decomposes into the direct product of
finite irreducible $p$-free Coxeter groups
$
W\cong W_1\times\cdots\times W_r
$
with $\Sigma_{i=1}^r\rank\,W_i=\rank\,W$.
So it suffices to prove the proposition when $W$ is irreducible.
Since $|W(A_n)|=(n+1)! $, $|W(B_n)|=2^n n!$ and $|W(D_n)|=2^{n-1} n!$,
they have no $p$-torsion when $n\leq p-2$. As for $W(I_2(m))$, it is $p$-free
if and only if it has no $p$-torsion.
These observations imply the proposition for $p\geq 11$,
for all finite irreducible Coxeter groups of type other than $A_n,B_n,D_n$ and $I_2(m)$
have no $p$-torsion for $p\geq 11$ (see Appendix).
Apart from Coxeter groups of type $A_n$, $B_n$, $D_n$ and $I_2(m)$,
finite irreducible Coxeter groups with rank at most $p-2$ are,
$W(H_3)$ for $p=5$, and $W(F_4)$, $W(H_3)$ and $W(H_4)$ for $p=7$.
But $W(H_3)$ is not $5$-free, while $W(F_4)$, $W(H_3)$ and $W(H_4)$ have
no $7$-torsion, proving the proposition.
\end{proof}
As an immediate consequence, we obtain the following corollary
which implies Theorem \ref{main-thm} for $|S|\leq p-2$:
\begin{cor}\label{vanish-p-2}
Let $p\geq 5$ be a prime and
$W$ a finite $p$-free Coxeter group with $|S|\leq p-2$.
Then $H^{k}(W,\F_p)=H^k(A_W,\F_p)=0$ for $k>0$.
\end{cor}

\section{Equivariant cohomology}\label{sec:equiv-coho}
Let $G$ be a group.
By a $G$-complex we mean a CW-complex $X$ together with a continuous action of
$G$ on $X$ which permutes the cells.
A $G$-complex $X$ is called admissible in \cite{brown} if,
for each cell $\sigma$ of $X$,
the isotropy subgroup $G_{\sigma}$ of $\sigma$ fixes $\sigma$ pointwise.
Throughout this section, $X$ is a finite dimensional, connected, admissible
$G$-complex, and $A$ is a trivial $G$-module
(an abelian group equipped with the trivial $G$-action).
We consider the equivariant cohomology $H^*_G(X,A)$
(see \cite[Chapter VII]{brown} for the definition and relevant facts).
\begin{lem}\label{lem:Serre-ss}
If $H^k(X,A)=0$ for $0<k<d$ then
$H^k_G(X,A)\cong H^k(G,A)$ for $0\leq k<d$.
\end{lem}
\begin{proof}
Consider the spectral sequence
\[
E_2^{ij}=H^i(G,H^j(X,A))\Rightarrow H^{i+j}_G(X,A)
\]
(see \cite[\S VII.7]{brown}).
We have $E_2^{*,0}=H^*(G,A)$ and $E_2^{*,j}=0$ for $0<j<d$.
This proves $H^k_G(X,A)\cong H^k(G,A)$ for $0\leq k<d$.
\end{proof}
Next, consider the spectral sequence
\begin{equation}\label{eq:Leray-ss}
E_1^{ij}=H^j(G,C^i(X,A))
\Rightarrow H^{i+j}_G(X,A).
\end{equation}
By Shapiro's lemma,
\begin{equation}\label{eq:E_1-Shapiro}
E_1^{ij}\cong\prod_{\sigma\in\mathcal{E}_i} H^j(G_\sigma,A)
\end{equation}
where $\mathcal{E}_i$ is a set of representatives for $G$-orbits of $i$-cells of $X$
(see \cite[\S VII.7]{brown}).
Note that  $A$ in $H^j(G_\sigma,A)$ is the trivial $G_\sigma$-module since $X$ is admissible.
\begin{lem}\label{lem:coho-orbit}
In the spectral sequence (\ref{eq:Leray-ss}), 
we have $E_2^{*,0}\cong H^*(X/G,A)$.
\end{lem}
\begin{proof}
Recall that the differential
$d_1^{ij}:E_1^{ij}\rightarrow E_1^{i+1,j}$ of the spectral sequence is
the map $H^j(G,C^i(X,A))\rightarrow H^j(G,C^{i+1}(X,A))$ induced
by the coboundary operator of $C^*(X,A)$.
Now there are isomorphisms
\[
H^0(G,C^i(X,A))\cong C^i(X,A)^G\cong C^i(X/G,A),
\]
where the second isomorphism holds because $X$ is admissible.
These isomorphisms are compatible with differentials of the spectral sequence
and coboundary operators of $C^*(X,A)$ and $C^*(X/G,A)$,
the lemma follows.
\end{proof}
\begin{lem}\label{lem:equivariant-vanish}
If $X$ satisfies the following conditions:
\begin{enumerate}
\item For each cell $\sigma$ of $X$, $H^k(G_\sigma,A)=0$ for $0<k<d$,
\item $H^k(X/G,A)=0$ for $0<k<d$.
\end{enumerate}
Then $H^k_G(X,A)=0$ for $0<k<d$.
\end{lem}
\begin{proof}
Consider the spectral sequence (\ref{eq:Leray-ss}).
We have $E_1^{*,j}=0$ for $0<j<d$ by the isomorphism (\ref{eq:E_1-Shapiro}),
and $E_2^{i,0}\cong H^i(X/G,A)=0$ for $0<i<d$
by Lemma \ref{lem:coho-orbit}, proving the lemma.
\end{proof}
Combining Lemma \ref{lem:Serre-ss} and \ref{lem:equivariant-vanish},
we obtain the following proposition which will be used to prove Theorem
\ref{main-thm}:
\begin{pro}\label{pro:equivariant}
Let $G$ be a group and $A$ a trivial $G$-module.
If there exists  a finite dimensional, admissible, connected
$G$-complex $X$ satisfying the
following conditions:
\begin{enumerate}
\item For each cell $\sigma$ of $X$, $H^k(G_\sigma,A)=0$ for $0<k<d$,
\item $H^k(X,A)=0$ and $H^k(X/G,A)=0$ for $0<k<d$.
\end{enumerate}
Then $H^k(G,A)=0$ for $0<k<d$.
\end{pro}

\section{Coxeter complexes and the proof of the main theorem}\label{complex}
Now we prove Theorem \ref{main-thm}.
To do so, first
we recall the definition and properties of Coxeter complexes
which are relevant to prove Theorem \ref{main-thm}.
A reference for Coxeter complexes is \cite{book-building}*{Chapter 3}.
Given a Coxeter group $W$, the \emph{Coxeter complex} 
$X_W$ of $W$ is the poset of
cosets $wW_T$ $(w\in W, T\subsetneq S)$,
ordered by reverse inclusion. 
It is known that $X_W$ is indeed an $(|S|-1)$-dimensional
simplicial complex (see \cite[Theorem 3.5]{book-building}). 
The $k$-simplices of $X_W$ are the cosets $wW_T$ with $k=|S|-|T|-1$.
A coset $wW_T$ is a face of $w'W_{T'}$ if and only if $wW_T\supseteq w'W_{T'}$.
In particular, the vertices are cosets of the form $wW_{S\setminus\{s\}}$
$(s\in S,w\in W)$, 
while the maximal simplices are the singletons $wW_\varnothing=\{w\}$ $(w\in W)$.
The maximal simplex $W_\varnothing=\{1\}$ is called the \emph{fundamental chamber}.
In what follows, we will not distinguish between $X_W$ and its
geometric realization.
The following fact is well-known (see \cite[Proposition 1.108]{book-building}).
\begin{pro}\label{pro:Coxeter-cpx}
If $W$ is a finite Coxeter group, then $X_W$ is a triangulation of the
$(|S|-1)$-dimensional sphere $\sph^{|S|-1}$.
\end{pro}
In case $W$ is infinite, Serre proved the following result:
\begin{pro}[{\cite[Lemma 4]{serre}}]\label{serre-cont}
If $W$ is an infinite Coxeter group, 
then $X_W$ is contractible.
\end{pro}
There is a simplicial action of $W$ on $X_W$ by left translation 
$w'\cdot wW_T:=w'wW_T$. The isotropy subgroup of a simplex $wW_T$ is
precisely $wW_T w^{-1}$ which
fixes $wW_T$ pointwise.
Hence $X_W$ is an admissible $W$-complex.
Now let $\Delta_W=\{W_T\ |\ T\subsetneq S\}$ be the subcomplex of $X_W$,
which consists of the fundamental chamber $W_\varnothing$ and its faces.
Then $\Delta_W$ is a set of representatives for the $W$-orbits of simplices,
and hence $\Delta_W$ is a strict fundamental domain for the action of $W$
(see \cite{book-building}*{Lemma 3.75}).

\begin{pro}\label{orbit-space}
For any Coxeter group $W$ not necessarily finite, 
the orbit space $X_W/A_W$ is homeomorphic
to the $(|S|-1)$-dimensional sphere $\sph^{|S|-1}$.
\end{pro}
\begin{proof}
Since $X_W$ is an admissible $A_W$-complex, the orbit space
$X_W/A_W$ inherits a CW-structure whose cells correspond bijectively to
$A_W$-orbits of simplices of $X_W$.
Pick $s_0\in S$ arbitrary. 
As $W=A_W\sqcup s_0A_W$ and $\Delta_W$ is a set of representatives
for the $W$-orbits of simplices, each $A_W$-orbit of simplices is represented
by either $W_T$ or $s_0W_T$ $(T\subsetneq S)$.
The isotropy subgroup of the fundamental chamber $W_\varnothing$ is trivial,
which implies $W_\varnothing$ and $s_0W_\varnothing$ represent distinct $A_W$-orbits.
On the other hand, for $\varnothing\not=T\subsetneq S$ and $t\in T$,
we have $ts_0\in A_W$ and $ts_0\cdot s_0 W_T=W_T$, which implies
$W_T$ and $s_0W_T$ represent the same $A_W$-orbit.
As a result, $X_W/A_W$ can be identified with the cell complex obtained from
$\Delta_W\sqcup s_0\Delta_W$ by identifying faces $W_T$ and $s_0W_T$
$(T\not\not=\varnothing)$.
We conclude that $X_W/A_W$ is homeomorphic to $\sph^{|S|-1}$.
\end{proof}

Now we prove Theorem \ref{main-thm} by induction on $|S|$.
We may assume  $|S|> p-2$ by Corollary \ref{vanish-p-2}.
Consider the action of $A_W$ on the Coxeter complex $X_W$.
We have $H^k(X_W,\F_p)=H^k(X_W/A_W,\F_p)=0$ for
$0<k<|S|-1$ by Proposition \ref{pro:Coxeter-cpx} and \ref{orbit-space},
and hence we have
\begin{equation}\label{eq:cond1}
H^k(X_W,\F_p)=H^k(X_W/A_W,\F_p)=0\ (0<k<p-2).
\end{equation}
Moreover, for each simplex $\sigma:=wW_T$ of $X_W$,
the isotropy subgroup $(A_W)_\sigma$ of $\sigma$ satisfies
\[
(A_W)_\sigma=wW_Tw^{-1}\cap A_W=w(W_T\cap A_W)w^{-1}=wA_{W_T}w^{-1}\cong A_{W_T}
\]
because $A_W$ is a normal subgroup of $W$.
Since $A_{W_T}$ is the alternating subgroup of $W_T$ with $|T|<|S|$,
we see that
\begin{equation}\label{eq:cond2}
H^k((A_W)_\sigma,\F_p)\cong H^k(A_{W_T},\F_p)=0\ (0<k<p-2)
\end{equation}
by the induction assumption.
Applying (\ref{eq:cond1}) and (\ref{eq:cond2}) to Proposition \ref{pro:equivariant},
Theorem \ref{main-thm} follows.

\section{Twisted cohomology of finite Coxeter groups}\label{sec:twisted}
Given a Coxeter system $(W,S)$, let $\signmod$ be 
the sign representation of $W$ over $\F_p$.
Namely, $\signmod =\F_p$ as an abelian group, and each $s\in S$ acts on
$\signmod$ as the multiplication by $-1$.
As an application of Theorem \ref{main-thm}, we will deduce the vanishing range for
$H^k(W,\signmod)$ as follows:
\begin{thm}\label{thm:signmod}
Let $W$ be a finite $p$-free Coxeter group.
Then $H^k(W,\signmod)=0$ for $k<p-2$.
\end{thm}
\begin{proof}
Observe that $H^0(W,\signmod)\cong \signmod^W=0$.
Let $M:=\mathrm{Ind}^W_{A_W}(\F_p)$ be the induced module
of the trivial $A_W$-module $\F_p$.
Then
$
H^k(W,M)\cong H^k(A_W,\F_p)
$
by Shapiro's lemma.
On the other hand, since $p$ is odd, 
$M$ decomposes into $M\cong\F_p\oplus\signmod$,
which implies
$H^k(W,M)\cong H^k(W,\F_p)\oplus H^k(W,\signmod)$
and hence
\begin{equation}\label{sign-split}
H^k(A_W,\F_p)\cong H^k(W,\F_p)\oplus H^k(W,\signmod).
\end{equation}
Now Theorem \ref{main-thm} implies $H^k(W,\signmod)=0$
for $0<k<p-2$.
\end{proof}
The proof also implies  $H^k(W,\F_p)=0$ holds for $0<k<p-2$,
however, the first author proved a much stronger result in
\cite{preprint-coxeter-p-local}.
Now let $\Sigma_p$ be the symmetric group on $p$ letters and $A_p$
the alternating group on $p$ letters.
It is known that
$H^{p-2}(\Sigma_p,\signmod)\cong\F_p$ (see \cite{MR1644252}*{pp.~74--75}).
So vanishing ranges in Theorem \ref{thm:signmod} are best possible.
Moreover, the isomorphism (\ref{sign-split}) implies
$H^{p-2}(A_p,\F_p)\not=0$, which shows vanishing ranges in Theorem \ref{main-thm}
are also best possible.

\section{Alternating subgroups of infinite Coxeter groups}\label{sec:infinite}
In general, Theorem \ref{main-thm} no longer holds for $p$-free Coxeter groups of infinite order.
For example, let $D_\infty=\langle s,t\mid s^2=t^2=1\rangle\cong\Z/2\ast\Z/2$ be the
infinite dihedral group. By definition, it is $p$-free for all $p$ since $m(s,t)=\infty$.
On the other hand, the alternating subgroup $A_{D_\infty}$ of $D_\infty$ is the infinite cyclic
group generated by $st\in D_\infty$ so that $H^1(A_{D_\infty},\F_p)\cong\F_p$ for all $p$.
More generally, we have the following proposition:
\begin{pro}
If $W$ is an infinite Coxeter group all of whose proper parabolic subgroups
are of finite order, then $H^{|S|-1}(A_W,\F_p)\not=0$ for all $p$.
\end{pro}
\begin{proof}
Note first $H^*(A_W,\Q)\cong H^*_{A_W}(X_W,\Q)$
because $X_W$ is contractible by Proposition \ref{serre-cont}.
Consider the spectral sequence (\ref{eq:Leray-ss}) and (\ref{eq:E_1-Shapiro}) with
$\Q$-coefficients:
\[
E_1^{ij}=\prod_{\sigma\in\mathcal{E}_i}H^j((A_W)_\sigma,\Q)
\Rightarrow H^{i+j}_{A_W}(X_W,\Q).
\]
Since the isotropy subgroup
$(A_W)_\sigma$ is finite for any $\sigma$ by the assumption, 
we see that $H^j((A_W)_\sigma,\Q)=0$ for $j>0$ and
hence $E_1^{*,j}=0$ for $j>0$.
In addition, $E_2^{*,0}\cong H^*(\sph^{|S|-1},\Q)$ by Proposition \ref{lem:coho-orbit}
and \ref{orbit-space}. As a result, we have 
\[H^*(A_W,\Q)\cong H^*(\sph^{|S|-1},\Q)\]
and hence $H^{|S|-1}(A_W,\F_p)\not=0$ for any prime $p$ in virtue of the
universal coefficient theorem.
\end{proof}
The assumption that all proper parabolic subgroups of $W$ are finite is somewhat restrictive.
Apart from finite Coxeter groups,
it holds if and only if (a) $W$ is an irreducible Euclidean reflection group or
(b) $W$ is a hyperbolic reflection group whose fundamental domain is a closed simplex contained
entirely in the interior of the hyperbolic space.
See \cite{book-building}*{Remark 3.29} for precise.
Now we prove the weak version of Theorem \ref{main-thm} for those Coxeter groups.
\begin{thm}\label{thm:infinite}
Let $W$ be an infinite $p$-free Coxeter group all of whose proper parabolic subgroups
are of finite order.
Then $H^k(A_W,\F_p)=0$ for $0<k<\min\{p-2,|S|-1\}$.
\end{thm}
\begin{proof}
Set $d=\min\{p-2,|S|-1\}$ and consider the Coxeter complex $X_W$.
For each $k$-simplex $\sigma$ of $X_W$, the isotropy subgroup $(A_W)_\sigma$
is isomorphic to $A_{W_T}$ for some $T\subsetneq S$ as in the proof of Theorem \ref{main-thm},
and hence
$H^k((A_W)_\sigma,\F_p)=0$ for $0<k<d$ by the assumption and Theorem
\ref{main-thm}.
On the other hand, noting $d\leq |S|-1$,
$H^k(X_W,\F_p)=H^k(X_W/A_W,\F_p)=0$ for $0<k<d$
by Proposition \ref{serre-cont} and \ref{orbit-space}.
Now the theorem follows from Proposition \ref{pro:equivariant}.
\end{proof}
The proof of the following corollary is similar to the one for Theorem \ref{thm:signmod}:
\begin{cor} Under the assumption of Theorem \ref{thm:infinite},
$H^k(W,\F_p[-1])=0$ holds for $k<\min\{p-2,|S|-1\}$.
\end{cor}

\section*{Appendix}
The following is the table for the Coxeter graph $\Gamma$, 
the order $|W(\Gamma)|$ of the corresponding 
finite irreducible Coxeter group $W(\Gamma)$, and
the range of odd prime numbers $p$ such that $W(\Gamma)$ is $p$-free.

\[
\begin{array}{ccc}
\Gamma & |W(\Gamma)| & \text{$p$-freeness} \\ \hline
A_1 & 2 & p\geq 3 \\
A_n\,(n\geq 2) & (n+1)! & p\geq 5 \\
B_2 & 8 & p\geq 3\\
B_n\,(n\geq 3) & 2^n n! & p\geq 5\\
D_n\,(n\geq 4) & 2^{n-1} n! & p\geq 5\\
E_6 &  2^7\cdot 3^4\cdot 5 & p\geq 5 \\
E_7 & 2^{10}\cdot 3^4\cdot 5\cdot 7 & p\geq 5 \\
E_8 & 2^{14}\cdot 3^5\cdot 5^2\cdot 7 & p\geq 5 \\
F_4 & 2^7\cdot 3^2 & p\geq 5 \\
H_3 & 2^3\cdot 3\cdot 5 & p\geq 7 \\
H_4 & 2^6\cdot 3^2\cdot 5^2 & p\geq 7 \\
I_2(m)\,(m\geq 3) & 2m & p\not|\ m
\end{array}
\]

\begin{ack}
The first author was partially supported by JSPS KAKENHI Grant Number
26400077.
\end{ack}





\end{document}